\documentclass[12pt]{amsart}

\usepackage{latexsym, amssymb, amsmath,ulem,soul,esint}

\usepackage{amsfonts, graphicx}
\usepackage{graphicx,color}
\newcommand{\be}{\begin{equation}}
\newcommand{\ee}{\end{equation}}
\newcommand{\beq}{\begin{eqnarray}}
\newcommand{\eeq}{\end{eqnarray}}

\newtheorem{thm}{Theorem}[section]

\newtheorem{lma}{Lemma}[section]
\newtheorem{prop}{Proposition}[section]
\newtheorem{cor}{Corollary}[section]

\newtheorem{claim}{Claim}[section]
\theoremstyle{remark}
\newtheorem{rem}{Remark}[section]
\numberwithin{equation}{section}

\def\be{\begin{equation}}
\def\ee{\end{equation}}
\def\bee{\begin{equation*}}
\def\eee{\end{equation*}}

\newcommand{\de}{\partial}

\newcommand{\Tmax}{T_{\textrm{max}}}

\newcommand{\Ric}{\mathrm{Ric}}

\newcommand{\ve}{\varepsilon}
\newcommand{\Vol}{\mathrm{Vol}}

\def\Ric{\text{\rm Ric}}
\def\Rm{\text{\rm Rm}}

\def\e{\varepsilon}
\def\a{{\alpha}}

\begin{document}

\title[]
{Conformal tori with almost non-negative scalar curvature }

 \author{Jianchun Chu}
\address[Jianchun Chu]{Department of Mathematics, Northwestern University, 2033 Sheridan Road, Evanston, IL 60208}
\email{jianchun@math.northwestern.edu}

 \author{Man-Chun Lee}
\address[Man-Chun Lee]{Mathematics Institute, Zeeman Building,
University of Warwick, Coventry CV4 7AL; Department of Mathematics, Northwestern University, 2033 Sheridan Road, Evanston, IL 60208}
\email{Man.C.Lee@warwick.ac.uk, mclee@math.northwestern.edu}

\renewcommand{\subjclassname}{
  \textup{2010} Mathematics Subject Classification}

\date{\today}

\begin{abstract}
In this work, we consider sequence of metrics with almost non-negative scalar curvature on torus. We show that if the sequence is uniformly conformal to another sequence of metrics with better controlled geometry, then it converges to a flat metric in the volume preserving intrinsic flat sense,  $L^{p}$ sense and the measured Gromov-Hausdorff sense.
\end{abstract}


\maketitle

\markboth{ Jianchun Chu and Man-Chun Lee}{}

\section{Introduction}
In the study of Riemnnian geometry, the notion of curvature plays an significant role. As the average of sectional curvature, the scalar curvature is one of the simplest curvature invariants on a Riemannian manifold. In general, the scalar curvature can be regarded as a kind of weak measure of the local geometry. It is tempting to ask which manifolds can admit metric of positive scalar curvature. When the underlying manifold is torus $\mathbb{T}^n$, the Geroch Conjecture predicted that metrics with non-negative scalar curvature must be flat. The problem was solved by Schoen-Yau \cite{SchoenYau1979,SchoenYau1979-2} for $n\leq 7$ using minimal surface method  and Gromov-Lawson \cite{GromovLawson1980} for general $n$ using Atiyah-Singer index theorem for a twisted spinor bundle on a spin manifold.

In \cite{Gromov2014},  Gromov  conjectured a stability for the torus rigidity. Namely, a sequence of Riemannian manifolds with almost non-negative scalar curvature, which are diffeomorphic to tori, combined with appropriate compactness conditions should converge to a flat torus in some weak sense. In \cite{Sormani2017}, Sormani had formulated the conjecture more concretely using the notion of intrinsic flat distance which is a distance between integral current spaces and was introduced by Sormani-Wenger \cite{SormaniWenger2011}. This is believed to be the suitable notion for taking limits of manifolds with lower scalar curvature bounds, see also the recent work of the second named author, Naber and Neumayer in \cite{LeeNaberNeumayer2020} which suggested that the geodesic distance should be replaced by the $L^p$ version of distance function $d_p$ at least when $n>3$. For general dimension, it is unclear what conditions should serve as the non-collapsing assumption. While in $n=3$, Sormani \cite{Sormani2017} gave a precise prediction of the non-collapsing conditions called the MinA condition in order to avoid bubbling occurring.

The first result in this direction is given by Gromov in \cite{Gromov2014}  where if one assumes that a sequence of tori with almost non-negative scalar curvature converges in the $C^0$ sense to a $C^2$ metric, then one can show that the $C^2$ limit is a flat Riemannian metric. In \cite{Bamler2016}, Bamler gave an alternative proof using the Ricci flow to perform regularization. It was recently generalized by Burkhardt-Guim \cite{Burkhardt2019} to the case when the limiting metrics $g_\infty$ are only $C^0$ metric and it was shown that $g_\infty$ is isometric to the flat torus as a metric space. Further progresses toward the conjecture formulated by Sormani \cite{Sormani2017} has been made in various cases. In \cite{AHPPS2019}, Allen, Hernandez-Vazquez, Parise, Payne, and Wang studied the warped product case. In \cite{CAP2020}, Cabrera Pacheco, Ketterer, and Perales studied the case of graphical tori. In \cite{Allen2020}, Allen studied the case when the sequence is conformal to the flat tori. In \cite{LeeNaberNeumayer2020}, the second named author with Naber and Neumayer considered the case when the entropy is small and proved the stability in the sense of $d_p$ convergence.

Motivated by the work of Allen in \cite{Allen2020}, we study the case when the metrics are uniformly conformal to some metrics with better controlled geometry. We show that under the natural non-collapsing assumption, if in addition the conformal factor is bounded in $L^{p_{0}}$ for some sufficiently large $p_{0}$, then the sequence will converge to a flat metric in the volume preserving intrinsic flat sense.

\begin{thm}\label{Sec1: maintheorem}
Let $(M_{i},g_{i})$ be a sequence of $n$-dimensional Riemannian manifolds ($n>2$). Suppose $M_{i}$ is diffeomorphic to $\mathbb{T}^n$ and $g_i=u_i^\frac4{n-2}h_i$ for some metric $h_i$ such that
\begin{enumerate}\setlength{\itemsep}{1mm}
\item[(i)]  $|\Ric(h_i)|_{h_i}+(\mathrm{inj}(M_i,h_i))^{-1}+\mathrm{diam}(M_i,h_i)\leq \Lambda$;
\item[(ii)] $\mathrm{diam}(M_i,g_i)+\|u_i\|_{L^{p_0}(M_i,h_i)}\leq \Lambda $;
\item[(iii)] $\mathrm{Vol}(M_i,g_i)\geq \Lambda^{-1}$;
\item[(iv)] $R(g_i)\geq -\delta_i$
\end{enumerate}
for some $\Lambda>0$, $p_0>\frac{n}{2}+\frac{2n(n+4)}{(n-2)(n+2)}$ and $\delta_i\rightarrow 0$. Then after passing to a subsequence, $(M_i,g_i)$ converges to a flat torus in the volume preserving intrinsic flat sense.
\end{thm}
\begin{rem}
By the H\"older inequality, condition (ii) and (iii) will imply a lower bound of $\mathrm{Vol}(M_i,h_i)$. In this case, the lower bound of injectivity radius of $h_i$ will follow from the work of Cheeger-Gromov-Taylor \cite{CheegerGromovTaylor1982} if we strengthen the curvature bound of $h_i$ from $\Ric(h_i)$ to $\Rm(h_i)$.
\end{rem}

In Theorem~\ref{Sec1: maintheorem}, $h_i$ serves as a sequence of well-behaved reference metrics within the conformal classes. In certain sense, we show that if the metric within the conformal has almost non-negative scalar curvature, then it is almost the solution to the Yamabe problem. The volume condition (iii) is necessary in order to prevent collapsing. Condition (ii) is stronger than volume non-expanding which plays the role to rule out bubbling. If we strengthen condition (ii) from $L^{p_{0}}$ to $L^\infty$, better convergence can be obtained. Note that in this case, the upper bound of $\mathrm{diam}(M_i,g_i)$ follows from condition (i).

\begin{thm}\label{Sec1: maintheorem2}
Let $(M_{i},g_{i})$ be a sequence of $n$-dimensional Riemannian manifolds ($n>2$). Suppose $M_{i}$ is diffeomorphic to $\mathbb{T}^n$ and $g_i=u_i^\frac4{n-2}h_i$ for some metric $h_i$ such that
\begin{enumerate}\setlength{\itemsep}{1mm}
\item[(i)]  $|\Ric(h_i)|_{h_i}+(\mathrm{inj}(M_i,h_i))^{-1}+\mathrm{diam}(M_i,h_i)\leq \Lambda$;
\item[(ii)] $ \sup_{M_i}u_i\leq\Lambda$;
\item[(iii)] $\mathrm{Vol}(M_i,g_i)\geq \Lambda^{-1}$;
\item[(iv)] $R(g_i)\geq -\delta_i$
\end{enumerate}
for some $\Lambda>0$ and $\delta_i\to 0$. Then after passing to a subsequence, $(M_i,g_i)$ converges to a flat torus in $L^p$ for all $p>0$ modulo diffeomorphism. Moreover, $(M_i,g_i)$ converges to a flat torus in the measured Gromov-Hausdorff sense.
\end{thm}

\begin{rem}
It will  be clear from the proof that the torus structure doesn't play any important role. Indeed, the same stability will be true if $\mathbb{T}^n$ is replaced by any closed manifolds $M^n$ with non-positive Yamabe invariant. We refer readers to Section~\ref{stability-T} for the detailed statement.
\end{rem}

The main technique is motivated by that of \cite{Bamler2016,Burkhardt2019,ShiTam2018} which used the Ricci flow to regularize the metrics using the estimates from \cite{KochLamm2012, Simon2002}. In our case, we do not a-priori assume any convergence of $g_i$ or $C^0$ closeness of metrics and therefore Ricci flow's method do not apply directly. Instead,  we make use of the Yamabe flow which is a geometric heat flow evolving inside the conformal class to (partially) regularize the metric $g_i$. Unlike the Ricci flow, its regularization ability is  relatively limited and highly depends on the uniform geometry of $h_i$. In our case, although the corresponding Yamabe flow $g_i(t)$ is not uniformly regular in $C^\infty$, we are able to show that it converges to a fixed flat metric away from $t=0$ weakly modulo diffeomorphism. This reduces the problem to establishing uniform weak convergence of $g_i(t)$ as $t\rightarrow 0$. This will be done in Section~\ref{stability-T}.

{\it Acknowledgement}:
We would like to thank Brian Allen and Shengwen Wang for useful communication. We would also like to thank Davi Maximo for pointing out a discrepancy in the earlier version of the paper. M.-C. Lee was partially supported by NSF grant DMS-1709894 and EPSRC  grant number P/T019824/1.

\section{A-priori estimates along the Yamabe flow}\label{estimateYamaflow}
In the following, we will consider a Riemannian manifold $(M,h)$ which satisfies
\begin{equation*}
{\bf (A)}:\quad\;\;
\left\{
\begin{array}{ll}
|\Ric(h)|_{h} \leq \Lambda;\\[1mm]
\mathrm{diam}(M,h)\leq \Lambda;\\[1mm]
\mathrm{inj}(M,h)\geq \Lambda^{-1}
\end{array}
\right.
\end{equation*}
for some $\Lambda>0$. For later use, here we recall the Sobolev inequality and the Poincar\'e inequality: for $f\in W^{1,2}(M)$,
\begin{equation}
\left(\int_{M}|f|^{\frac{2n}{n-2}}d\mu_{h}\right)^{\frac{n-2}{n}}
\leq C_{S}\int_{M}|\nabla^{h}f|_{h}^{2}d\mu_{h}+C_{S}\int_{M}f^{2}d\mu_{h},
\end{equation}
and
\begin{equation}
\int_{M}|f-\underline{f}|^{2}d\mu_{h}
\leq C_{P}\int_{M}|\nabla^{h}f|_{h}^{2}d\mu_{h}, \quad \underline{f} = \frac{1}{\Vol(M,h)}\int_{M}fd\mu_{h}.
\end{equation}
Note that the lower bound of injective radius implies the lower bound of volume (see \cite{BergerKazdan80}), and then constants $C_{S}$ (see \cite{Ilias83}) and $C_{P}$ (see \cite{LiYau80}) depend only on $n$ and $\Lambda$.

Let $g_0=u_0^\frac4{n-2}h$ be a metric inside the conformal class of $h$. Suppose that
\begin{equation*}
{\bf (B)}:\quad\;\; \left\{
\begin{array}{ll}
\|u_0\|_{L^{p_0}(M,h)}\leq \Lambda;\\[1.5mm]
\mathrm{Vol}(M,g_{0})\geq \Lambda^{-1};\\[1mm]
R(g_0)\geq -\delta\geq -1
\end{array}
\right.
\end{equation*}
for some $p_0,\Lambda,\delta>0$. In this section, we will use the Yamabe flow to regularize $g_0$ slightly. This is the family of metric $g(t)$ which solves
\begin{equation}\label{Yamabe flow}
\left\{
\begin{array}{ll}
\frac{\de g}{\de t} = -R(g) g,\\[1mm]
g(0) = g_0.
\end{array}
\right.
\end{equation}
Equivalently, if we write $g(t)=u(t)^{\frac{4}{n-2}}h$, then the function $u(t)$ solves
\begin{equation}\label{evo-u}
\begin{cases}
\ \frac{\de u^N}{\de t} = \frac{n+2}{4}\left(\frac{4(n-1)}{n-2}\Delta_{h}u-R_{h}u \right), \\
\ u(0) = u_{0},
\end{cases}
\end{equation}
where $N=\frac{n+2}{n-2}$ and $R_{h}$ denotes the scalar curvature of $h$. We will establish a-priori estimates of $g(t)$ or equivalently $u(t)$ along the Yamabe flow.

\subsection{Lower bound of the Yamabe flow}
We first show that if $p_{0}$ is sufficiently large, then assumption {\bf (B)} implies the positive lower bound of $u_{0}$. We begin with the following lemma, which is a global version of \cite[(4.8)]{HanLin1997}.
\begin{lma}\label{Sec2: inequality}
If $p_{0}>\frac{2n}{n-2}$, then there is a constant $\ve_{0}(n,\Lambda,p_{0})>0$ such that
\begin{equation}
\left(\int_{M}u_{0}^{\ve_{0}}d\mu_{h}\right)\left(\int_{M}u_{0}^{-\ve_{0}}d\mu_{h}\right) \leq 4e^{2}.
\end{equation}
\end{lma}
\begin{proof}
The proof is similar to that of the local version \cite[(4.8)]{HanLin1997}. For the sake of completeness, we include the proof here. Since $g_{0}=u_{0}^{\frac{4}{n-2}}h$, then the scalar curvature of $g_{0}$ is
\begin{equation}\label{Sec2: inequality eqn 5}
R_{g_{0}} = u_{0}^{-\frac{n+2}{n-2}}\left(R_{h}u_{0}-\frac{4(n-1)}{n-2}\Delta_{h} u_{0}\right),
\end{equation}
Define
\begin{equation}\label{def of w}
w = \log u_{0}-\frac{1}{\mathrm{Vol}(M,h)}\int_{M}(\log u_{0})d\mu_{h}.
\end{equation}
Direct calculation shows
\begin{equation}
\begin{split}
R_{g_{0}} = {} & u_{0}^{-\frac{4}{n-2}}\left(R_{h}-\frac{4(n-1)}{n-2}\frac{\Delta_{h}u_{0}}{u_{0}}\right) \\
= {} & u_{0}^{-\frac{4}{n-2}}\left(R_{h}-\frac{4(n-1)}{n-2}\left(\Delta_{h}w+|\nabla^{h}w|_{h}^{2}\right)\right).
\end{split}
\end{equation}
Recalling $R_{g_{0}}\geq-1$,
\begin{equation}\label{Sec2: inequality eqn 1}
|\nabla^{h}w|_{h}^{2} = -\Delta_{h}w+\frac{n-2}{4(n-1)}\left(R_{h}-R_{g_{0}}u_{0}^{\frac{4}{n-2}}\right)
\leq -\Delta_{h}w+Cu_{0}^{\frac{4}{n-2}}+C
\end{equation}
for some $C(n,\Lambda)$. Since $p_{0}>\frac{2n}{n-2}>\frac{4}{n-2}$, we obtain
\begin{equation}\label{Sec2: inequality eqn 4}
\int_{M}|\nabla^{h}w|_{h}^{2}d\mu_{h} \leq \int_{M}(Cu_{0}^{\frac{4}{n-2}}+C)d\mu_{h}
\leq C\|u_{0}\|_{p_{0}}^{p_{0}}+C \leq C(n,\Lambda_{0},p_{0}).
\end{equation}
Combining this with the Sobolev inequality and the Poincar\'{e} inequality,
\begin{equation}\label{Sec2: inequality eqn 2}
\begin{split}
\left(\int_{M}|w|^{\frac{2n}{n-2}}d\mu_{h}\right)^{\frac{n-2}{n}}
\leq {} & C_{S}\int_{M}|\nabla^{h}w|_{h}^{2}d\mu_{h}+C_{S}\int_{M}|w|^{2}d\mu_{h} \\
\leq {} & C_{S}(1+C_{P})\int_{M}|\nabla^{h}w|_{h}^{2}d\mu_{h} \leq C(n,\Lambda,p_{0}).
\end{split}
\end{equation}
Define two constants $\beta$ and $\beta^{*}$ by
\begin{equation}
\frac{4\beta}{n-2} = p_{0}, \quad \frac{1}{\beta}+\frac{1}{\beta^{*}} = 1.
\end{equation}
Then $p_{0}>\frac{2n}{n-2}$ implies $\beta>\frac{n}{2}$ and $\beta^{*}<\frac{n}{n-2}$. Using \eqref{Sec2: inequality eqn 1}, \eqref{Sec2: inequality eqn 4} and the H\"older inequality, for $p>1$, we compute
\begin{equation}
\begin{split}
& \int_{M}|w|^{p}|\nabla^{h}w|_{h}^{2}d\mu_{h} \\
\leq {} & p\int_{M}|w|^{p-1}|\nabla^{h}w|_{h}^{2}d\mu_{h}+C\int_{M}|w|^{p}u_{0}^{\frac{4}{n-2}}d\mu_{h}+C\int_{M}|w|^{p}d\mu_{h} \\
\leq {} & \frac{1}{2}\int_{M}|w|^{p}|\nabla^{h}w|_{h}^{2}d\mu_{h}+(Cp)^{p}\int_{M}|\nabla^{h}w|_{h}^{2}d\mu_{h} \\
&+C\left(\int_{M}u_{0}^{p_{0}}d\mu_{h}\right)^{\frac{1}{\beta}}\left(\int_{M}|w|^{p\beta^{*}}d\mu_{h}\right)^{\frac{1}{\beta^{*}}}
+C\left(\int_{M}|w|^{p\beta^{*}}d\mu_{h}\right)^{\frac{1}{\beta^{*}}} \\
\leq {} & \frac{1}{2}\int_{M}|w|^{p}|\nabla^{h}w|_{h}^{2}d\mu_{h}+(Cp)^{p}+C\left(\int_{M}|w|^{p\beta^{*}}d\mu_{h}\right)^{\frac{1}{\beta^{*}}},
\end{split}
\end{equation}
which implies
\begin{equation}
\int_{M}|\nabla(|w|^{\frac{p+2}{2}})|_{h}^{2}d\mu_{h} \leq C(p+2)^{2}\left(\int_{M}|w|^{p\beta^{*}}d\mu_{h}\right)^{\frac{1}{\beta^{*}}}+(C(p+2))^{p+2},
\end{equation}
Combining this with the Sobolev inequality and the H\"older inequality,
\begin{equation}
\begin{split}
& \left(\int_{M}|w|^{\frac{n(p+2)}{n-2}}d\mu_{h}\right)^{\frac{n-2}{n}} \\
\leq {} & C(p+2)^{2}\left(\int_{M}|w|^{p\beta^{*}}d\mu_{h}\right)^{\frac{1}{\beta^{*}}}+C\int_{M}|w|^{p+2}d\mu_{h}+(C(p+2))^{p+2} \\
\leq {} & C(p+2)^{2}\left(\int_{M}|w|^{(p+2)\beta^{*}}d\mu_{h}\right)^{\frac{1}{\beta^{*}}}+(C(p+2))^{p+2}.
\end{split}
\end{equation}
It then follows that
\begin{equation}
\|w\|_{L^{\frac{n(p+2)}{n-2}}} \leq (C(p+2))^{\frac{2}{p+2}}\|w\|_{L^{(p+2)\beta^{*}}}+C(p+2), \quad \text{for $p>1$}.
\end{equation}
Replacing $(p+2)\beta^{*}$ by $p$, and writing $\gamma=\frac{n}{\beta^{*}(n-2)}$, we obtain
\begin{equation}\label{Sec2: inequality eqn 3}
\|w\|_{L^{\gamma p}} \leq (Cp)^{\frac{C}{p}}\|w\|_{L^{p}}+Cp, \quad \text{for $p>3\beta^{*}$}.
\end{equation}
Recalling $\beta^{*}<\frac{n}{n-2}$, we have $\gamma>1$. Using \eqref{Sec2: inequality eqn 3} repeatedly,
\begin{equation}\label{Sec2: inequality eqn 6}
\|w\|_{L^{6\beta^{*}}} \leq C\|w\|_{L^{4\beta^{*}}}+C.
\end{equation}
By \eqref{Sec2: inequality eqn 2}, we have $\|w\|_{L^{2\beta^{*}}}\leq C$. Using the Cauchy-Schwarz inequality,
\begin{equation}\label{Sec2: inequality eqn 7}
\|w\|_{L^{4\beta^{*}}} = \left(\int_{M}|w|^{\beta^{*}}|w|^{3\beta^{*}}d\mu_{h}\right)^{\frac{1}{4\beta^{*}}}
\leq  \|w\|_{L^{2\beta^{*}}}^{\frac{1}{4}}\|w\|_{L^{6\beta^{*}}}^{\frac{3}{4}}
\leq  C\|w\|_{L^{6\beta^{*}}}^{\frac{3}{4}}.
\end{equation}
Then \eqref{Sec2: inequality eqn 6} and \eqref{Sec2: inequality eqn 7} show $\|w\|_{L^{6\beta^{*}}}\leq C$. Combining this with \eqref{Sec2: inequality eqn 3}, we obtain
\begin{equation}
\|w\|_{L^{\gamma^{k}}} \leq \left(C\gamma^{k-1}\right)^{\frac{C}{\gamma^{k-1}}}\|w\|_{L^{\gamma^{k-1}}}+C\gamma^{k-1}, \quad \text{for $k\in\mathbb{N}$},
\end{equation}
and so
\begin{equation}
\|w\|_{L^{\gamma^{k}}} \leq C(n,\Lambda,p_{0})\gamma^{k}, \quad \text{for $k\in\mathbb{N}$}.
\end{equation}
This shows
\begin{equation}
\|w\|_{L^{p}} \leq C_{0}(n,\Lambda,p_{0})p,\quad \text{for $p\geq1$}.
\end{equation}
Choosing $\ve_{0}(n,\Lambda,p_{0})$ such that $C_{0}\ve_{0}\leq\frac{1}{2}$, then
\begin{equation}
\int_{M}e^{\ve_{0}|w|}d\mu_{h}
= \sum_{k=0}^{\infty}\int_{M}\frac{\ve_{0}^{k}|w|^{k}}{k!}d\mu_{h}
\leq \sum_{k=0}^{\infty}\frac{(C_{0}\ve_{0})^{k}k^{k}}{k!}
\leq e\sum_{k=0}^{\infty}(C_{0}\ve_{0})^{k} \leq 2e.
\end{equation}
Recalling the definition of $w$ \eqref{def of w}, we obtain
\begin{equation}
\left(\int_{M}u_{0}^{\ve_{0}}d\mu_{h}\right)\left(\int_{M}u_{0}^{-\ve_{0}}d\mu_{h}\right)
= \left(\int_{M}e^{\ve_{0}w}d\mu_{h}\right)\left(\int_{M}e^{-\ve_{0}w}d\mu_{h}\right) \leq 4e^{2}.
\end{equation}
\end{proof}

\begin{lma}\label{Sec2: low}
If $p_0>\frac{2n}{n-2}$, then there is a constant $c(n,\Lambda,p_0)>0$ such that
\begin{equation}
\inf_M u_0(x) \geq c(n,\Lambda,p_0).
\end{equation}
\end{lma}

\begin{proof}
By \eqref{Sec2: inequality eqn 5}, we have
\begin{equation}
\frac{4(n-1)}{n-2}\Delta_{h}u_{0} = R_{h}u_{0}-R_{g_0}u_{0}^\frac{n+2}{n-2}.
\end{equation}
Using $R_{g_{0}}\geq-1$, the function $v=u_{0}^{-1}$ satisfies
\begin{equation}\label{Sec2: low eqn 1}
\begin{split}
\Delta_{h} v = {} & -v^2\Delta_{h}u_{0}+2v^3 |\nabla^{h}u|_{h}^2 \\[1mm]
\geq {} & -\frac{(n-2)v}{4(n-1)} \left(R_h -R_{g_0}u_{0}^\frac{4}{n-2}\right) \\
\geq {} & -C(n)\Lambda v-C(n)u_{0}^\frac{4}{n-2}.
\end{split}
\end{equation}
Define two constants $\beta$ and $\beta^{*}$ by
\begin{equation}
\frac{4\beta}{n-2} = p_{0}, \quad \frac{1}{\beta}+\frac{1}{\beta^{*}} = 1.
\end{equation}
Then $p_{0}>\frac{2n}{n-2}$ implies $\beta>\frac{n}{2}$ and $\beta^{*}<\frac{n}{n-2}$. For $p>1$, multiplying $-v^{p-1}$ on both sides of \eqref{Sec2: low eqn 1} and integrating on $(M,h)$, we have
\begin{equation}\label{Sec2:low lemma}
\begin{split}
 \frac{4(p-1)}{p^2}\int_M | \nabla^{h}v^{\frac{p}{2}}|_h^2 d\mu_h = {} & -\int_M (v^{p-1} \Delta_h v) d\mu_h  \\[1mm]
\leq {} & C \int_M v^p d\mu_h +C \int_M (u_{0}^\frac4{n-2} v^p) d\mu_h\\
\leq {} & C \int_M v^p d\mu_h +C \left(\int_M u_{0}^{p_{0}} d\mu_h\right)^{\frac{1}{\beta}} \left(\int_M v^{p\beta^*}d\mu_h \right)^{\frac{1}{\beta^{*}}}
\end{split}
\end{equation}
for some $C(n,\Lambda)$. Combining this with $\|u_{0}\|_{L^{p_{0}}}\leq\Lambda$, the Sobolev inequality and the H\"older inequality, for $p>1$,
\begin{equation}
\begin{split}
\frac{p-1}{p^2} \left(\int_M v^\frac{pn}{n-2} d\mu_h \right)^\frac{n-2}{n}
&\leq C(n,\Lambda,p_0) \left(\int_M v^{p\beta^*}d\mu_h\right)^{1/\beta^*}.
\end{split}
\end{equation}
Recalling $\beta^*<\frac{n}{n-2}$, we may apply the iteration method to conclude that
\begin{equation}\label{Sec2: low-MV}
\|v\|_{L^\infty} \leq C(n,\Lambda,p_0) \|v\|_{L^{\ve_{0}}},
\end{equation}
where $\ve_{0}$ is the constant in Lemma \ref{Sec2: inequality}.

It remains to control $L^{\ve_{0}}$ bound of $v=u_{0}^{-1}$. Thanks to Lemma \ref{Sec2: inequality}, it suffices to establish the positive lower bound of $\|u_{0}\|_{L^{\ve_{0}}}$. We may assume $\e_0=\frac{2n\sigma }{n-2}$ for some $\sigma<1$. Otherwise, the required estimate follows from $\Vol(M,g_{0})\geq\Lambda^{-1}$ and the H\"older inequality. For any $\gamma\in(0,\sigma)$,
\begin{equation}\label{Sec2: rev--low}
\begin{split}
\Lambda^{-1} \leq {} & \mathrm{Vol}(M,g_{0}) = \int_M u_{0}^\frac{2n}{n-2} d\mu_h
=  \int_M u_{0}^{\frac{\ve_{0}\gamma}{\sigma}+\frac{2n(1-\gamma)}{n-2}} d\mu_h \\[1mm]
\leq {} & \left(\int_M u_{0}^{\e_0}d\mu_h \right)^\frac{\gamma}{\sigma}
\left(\int_M u_{0}^\frac{2n\sigma(1-\gamma)}{(n-2)(\sigma-\gamma)}d\mu_{h}\right)^\frac{\sigma-\gamma}{\sigma}.
\end{split}
\end{equation}
Since $p_0>\frac{2n}{n-2}$, we may choose $\gamma<\sigma$ such that
\begin{equation}
\frac{2n\sigma(1-\gamma)}{(n-2)(\sigma-\gamma)} = p_0.
\end{equation}
Using $\|u_{0}\|_{L^{p_{0}}}\leq\Lambda$, we obtain
\begin{equation}
\int_M u_{0}^{\e_0}d\mu_h \geq c(n,\Lambda,p_{0}) > 0,
\end{equation}
which completes the proof.
\end{proof}

Using Lemma \ref{Sec2: low}, we establish the lower bound of the Yamabe flow:
\begin{prop}\label{Sec2: low YF}
Let $g(t)=u(t)^{\frac{4}{n-2}}h$ be a solution of the Yamabe flow \eqref{Yamabe flow}. Suppose that $g(t)$ exists on $M\times[0,T]$. If $p_{0}>\frac{2n}{n-2}$, then there are constants $\hat{T}$ and $c$ depending only on $(n,\Lambda,p_{0})$ such that
\[
\inf_{M\times[0,\hat{T}\wedge T]}u \geq c > 0.
\]
\end{prop}

\begin{proof}
Applying the minimum principle to \eqref{evo-u} and using Lemma \ref{Sec2: low}, for $(x,t)\in M\times[0,T]$,
\begin{equation}
u(x,t) \geq \left(-C(n)\Lambda t+\inf_M u_0^\frac{4}{n-2}\right)^\frac{n-2}{4}
\geq \left(-C(n)\Lambda t+c_0(n,\Lambda,p_0)\right)^\frac{n-2}{4}.
\end{equation}
Choosing $\hat{T}=\frac{c_{0}}{2C(n)\Lambda}$, we complete the proof.
\end{proof}

\subsection{Upper bound of the Yamabe flow}
Next, we will show that if $p_0$ is sufficiently large, then along the Yamabe flow, $u(t)$ will be bounded from above instantaneously.
\begin{prop}\label{Sec2: up YF}
Let $g(t)=u(t)^{\frac{4}{n-2}}h$ be a solution of the Yamabe flow \eqref{Yamabe flow}. Suppose that $g(t)$ exists on $M\times[0,T]$ for some $T\leq1$. If
$\frac{n}{2}+\frac{2n(n+4)}{(n-2)(n+2)}<p_0<\infty$, then for any $\left(\frac{2p_{0}}{n}-\frac{4(n+4)}{(n-2)(n+2)}\right)^{-1}<\alpha<1$, there is a constant $C(\alpha,n,\Lambda,p_{0})$ such that
\[
\sup_{M}u(\cdot,t) \leq Ct^{-\alpha}, \quad \text{for $t\in[0,T]$}.
\]
If $p_{0}=\infty$, then there is a constant $C(n,\Lambda)$ such that
\[
\sup_{M\times[0,T]}u \leq C.
\]
\end{prop}
\begin{proof}
When $\frac{n}{2}+\frac{2n(n+4)}{(n-2)(n+2)}<p_0<\infty$, we split the proof into two steps:

\bigskip
\noindent
{\bf Step 1.} $\int_{0}^{T}\int_{M}u^{\frac{(n+2)p_{0}}{n}-N+1}d\mu_{h}dt\leq C(n,\Lambda,p_{0})$.
\bigskip

It is clear that $p_{0}>N=\frac{n+2}{n-2}$. Multiplying both sides of \eqref{evo-u} by $u^{p_{0}-N}$ and integrating by parts,
\begin{equation}
\begin{split}
& \frac{N}{p_{0}}\frac{\de}{\de t}\left(\int_{M}u^{p_{0}}d\mu_{h}\right)
+\frac{4N(n-1)(p_{0}-N)}{(p_{0}-N+1)^{2}}\int_{M}|\nabla^{h}u^{\frac{p_{0}-N+1}{2}}|_{h}^{2}d\mu_{h} \\
= {} & -\frac{n+2}{4}\int_{M}R_{h}u^{p_{0}-N+1}d\mu_{h}
\leq C(n)\Lambda\int_{M}u^{p_{0}-N+1}d\mu_{h}.
\end{split}
\end{equation}
Integrating this on $[0,T]$ and using $\|u_{0}\|_{L^{p_{0}}}\leq\Lambda$, we obtain
\begin{equation}
\begin{split}
& \sup_{t\in[0,T]}\left(\int_{M}u(t)^{p_{0}}d\mu_{h}\right)+\int_{0}^{T}\int_{M}|\nabla^{h}u^{\frac{p_{0}-N+1}{2}}|_{h}^{2}d\mu_{h}dt \\
\leq {} & C(n)\Lambda\int_{0}^{T}\int_{M}u^{p_{0}-N+1}d\mu_{h}dt+C(n,\Lambda,p_{0}).
\end{split}
\end{equation}
Combining this with the Sobolev inequality,
\begin{equation}
\begin{split}
& \sup_{t\in[0,T]}\left(\int_{M}u(t)^{p_{0}}d\mu_{h}\right)+\int_{0}^{T}\left(\int_{M}u^{\frac{n(p_{0}-N+1)}{n-2}}d\mu_{h}\right)^{\frac{n-2}{n}}dt \\
\leq {} & C(n)\Lambda\int_{0}^{T}\int_{M}u^{p_{0}-N+1}d\mu_{h}dt+C(n,\Lambda,p_{0}).
\end{split}
\end{equation}
Using $N>1$ and the Young inequality,
\begin{equation}
C(n)\Lambda\int_{0}^{T}\int_{M}u^{p_{0}-N+1}d\mu_{h}dt \leq \frac{1}{2}\sup_{t\in[0,T]}\left(\int_{M}u(t)^{p_{0}}d\mu_{h}\right)+C(n,\Lambda,p_{0}).
\end{equation}
Then
\begin{equation}\label{Sec2: up YF eqn 2}
\sup_{t\in[0,T]}\left(\int_{M}u(t)^{p_{0}}d\mu_{h}\right)+\int_{0}^{T}\left(\int_{M}u^{\frac{n(p_{0}-N+1)}{n-2}}d\mu_{h}\right)^{\frac{n-2}{n}}dt
\leq C(n,\Lambda,p_{0}).
\end{equation}
By the H\"older inequality,
\begin{equation}
\begin{split}
& \int_{0}^{T}\int_{M}u^{\frac{(n+2)p_{0}}{n}-N+1}d\mu_{h}dt
= \int_{0}^{T}\int_{M}u^{(p_{0}-N+1)+\frac{2p_{0}}{n}}d\mu_{h}dt \\
\leq {} & \left(\int_{0}^{T}\left(\int_{M}u^{\frac{n(p_{0}-N+1)}{n-2}}d\mu_{h}\right)^{\frac{n-2}{n}}dt\right)
\sup_{t\in[0,T]}\left(\int_{M}u^{p_{0}}(t)d\mu_{h}\right)^{\frac{2}{n}}
\leq C(n,\Lambda,p_{0}).
\end{split}
\end{equation}

\bigskip
\noindent
{\bf Step 2.} $\sup_{M}u(\cdot,t) \leq C(\alpha,n,\Lambda,p_{0})t^{-\alpha}$.
\bigskip

Define $v=t^{\alpha}u$. For $p>\max(\alpha^{-1},2)$, multiplying both sides of \eqref{evo-u} by $t^{\alpha}v^{p-1}$,
\begin{equation}
t^{\alpha}v^{p-1}\frac{\de u^{N}}{\de t}-N(n-1)v^{p-1}\Delta_{h}v = -\frac{n+2}{4}R_{h}v^{p}.
\end{equation}
We compute the first term of the left hand side:
\begin{equation}
\begin{split}
t^{\alpha}v^{p-1}\frac{\de u^{N}}{\de t}
= {} &  t^{\alpha p}u^{p-1}\frac{\de u^{N}}{\de t} = \frac{Nt^{\alpha p}}{N+p-1}\frac{\de u^{N+p-1}}{\de t} \\
= {} & \frac{N}{N+p-1}\frac{\de}{\de t}\left(t^{\alpha p}u^{N+p-1}\right)-\frac{\alpha Np}{N+p-1}t^{\alpha p-1}u^{N+p-1} \\
= {} & \frac{N}{N+p-1}\frac{\de}{\de t}\left(v^{p}u^{N-1}\right)-\frac{\alpha Np}{N+p-1}v^{p-\frac{1}{\alpha}}u^{N+\frac{1}{\alpha}-1}.
\end{split}
\end{equation}
Then
\begin{equation}
\begin{split}
& \frac{N}{N+p-1}\frac{\de}{\de t}\left(v^{p}u^{N-1}\right)-N(n-1)v^{p-1}\Delta_{h}v \\[1mm]
= {} & -\frac{n+2}{4}R_{h}v^{p}+\frac{\alpha Np}{N+p-1}v^{p-\frac{1}{\alpha}}u^{N+\frac{1}{\alpha}-1} \\[1mm]
\leq {} & C(n)\Lambda v^{p}+C(\alpha,n)v^{p-\frac{1}{\alpha}}u^{N+\frac{1}{\alpha}-1}.
\end{split}
\end{equation}
Integrating both sides on $(M,h)$,
\begin{equation}
\begin{split}
& \frac{N}{N+p-1}\frac{\de}{\de t}\left(\int_{M}v^{p}u^{N-1}d\mu_{h}\right)+\frac{4N(n-1)(p-1)}{p^{2}}\int_{M}|\nabla^{h}v^{\frac{p}{2}}|_{h}^{2}d\mu_{h} \\
\leq {} & C(\alpha,n,\Lambda)\int_{M}\left(v^{p}+v^{p-\frac{1}{\alpha}}u^{N+\frac{1}{\alpha}-1}\right)d\mu_{h}.
\end{split}
\end{equation}
Recalling that $v(0)=0$, the above inequality shows
\begin{equation}
\begin{split}
& \sup_{t\in[0,T]}\left(\int_{M}v^{p}(t)u^{N-1}(t)d\mu_{h}\right)+\int_{0}^{T}\int_{M}|\nabla^{h}v^{\frac{p}{2}}|_{h}^{2}d\mu_{h}dt \\
\leq {} & C(\alpha,n,\Lambda)p\int_{0}^{T}\int_{M}\left(v^{p}+v^{p-\frac{1}{\alpha}}u^{N+\frac{1}{\alpha}-1}\right)d\mu_{h}dt.
\end{split}
\end{equation}
By the Sobolev inequality,
\begin{equation}
\begin{split}
& \sup_{t\in[0,T]}\left(\int_{M}v^{p}(t)u^{N-1}(t)d\mu_{h}\right)+\int_{0}^{T}\left(\int_{M}v^{\frac{np}{n-2}}d\mu_{h}\right)^{\frac{n-2}{n}}dt \\
\leq {} & C(\alpha,n,\Lambda)p\int_{0}^{T}\int_{M}\left(v^{p}+v^{p-\frac{1}{\alpha}}u^{N+\frac{1}{\alpha}-1}\right)d\mu_{h}dt.
\end{split}
\end{equation}
Since $T\leq1$, then $v=t^{\alpha}u\leq u$, and so
\begin{equation}
\begin{split}
& \sup_{t\in[0,T]}\left(\int_{M}v^{p+N-1}(t)d\mu_{h}\right)+\int_{0}^{T}\left(\int_{M}v^{\frac{np}{n-2}}d\mu_{h}\right)^{\frac{n-2}{n}}dt \\
\leq {} & C(\alpha,n,\Lambda)p\int_{0}^{T}\int_{M}\left(v^{p}+v^{p-\frac{1}{\alpha}}u^{N+\frac{1}{\alpha}-1}\right)d\mu_{h}dt.
\end{split}
\end{equation}
Using $N>1$, we obtain
\begin{equation}
\begin{split}
& \sup_{t\in[0,T]}\left(\int_{M}v^{p}(t)d\mu_{h}\right)+\int_{0}^{T}\left(\int_{M}v^{\frac{np}{n-2}}d\mu_{h}\right)^{\frac{n-2}{n}}dt \\
\leq {} & Cp\int_{0}^{T}\int_{M}\left(v^{p}+v^{p-\frac{1}{\alpha}}u^{N+\frac{1}{\alpha}-1}\right)d\mu_{h}dt+C
\end{split}
\end{equation}
for some $C(\alpha,n,\Lambda)$. Combining this with the H\"older inequality,
\begin{equation}
\begin{split}
& \int_{0}^{T}\int_{M}v^{\frac{(n+2)p}{n}}d\mu_{h}dt
= \int_{0}^{T}\int_{M}v^{p+\frac{2p}{n}}d\mu_{h}dt \\
\leq {} & \left(\int_{0}^{T}\left(\int_{M}v^{\frac{np}{n-2}}d\mu_{h}\right)^{\frac{n-2}{n}}dt\right)
\sup_{t\in[0,T]}\left(\int_{M}v^{p}(t)d\mu_{h}\right)^{\frac{2}{n}} \\
\leq {} & \left(Cp\int_{0}^{T}\int_{M}\left(v^{p}+v^{p-\frac{1}{\alpha}}u^{N+\frac{1}{\alpha}-1}\right)d\mu_{h}dt+C\right)^{\frac{n+2}{n}}.
\end{split}
\end{equation}
Then
\begin{equation}\label{Sec2: up YF eqn 1}
\begin{split}
& \left(\int_{0}^{T}\int_{M}v^{\frac{(n+2)p}{n}}d\mu_{h}dt\right)^{\frac{n}{n+2}}
\leq Cp\int_{0}^{T}\int_{M}(v^{p}+1)(u^{N+\frac{1}{\alpha}-1}+1)d\mu_{h}dt+C \\
\leq {} & Cp\left(\left(\int_{0}^{T}\int_{M}v^{\beta p}d\mu_{h}dt\right)^{\frac{1}{\beta}}+1\right)
\left(\left(\int_{0}^{T}\int_{M}u^{\frac{(n+2)p_{0}}{n}-N+1}d\mu_{g}dt\right)^{\frac{1}{\beta^{*}}}+1\right),
\end{split}
\end{equation}
for some $C(\alpha,n,\Lambda)$, where
\begin{equation}
\frac{1}{\beta}+\frac{1}{\beta^{*}} = 1, \quad \beta^{*} = \frac{\frac{(n+2)p_{0}}{n}-N+1}{N+\frac{1}{\alpha}-1}.
\end{equation}
Combining \eqref{Sec2: up YF eqn 1} with Step 1,
\begin{equation}
\left(\int_{0}^{T}\int_{M}v^{\frac{(n+2)p}{n}}d\mu_{h}dt\right)^{\frac{n}{n+2}}
\leq  Cp\left(\int_{0}^{T}\int_{M}v^{\beta p}d\mu_{h}dt\right)^{\frac{1}{\beta}}+Cp
\end{equation}
for some $C(\alpha,n,\Lambda,p_{0})$. Recalling $\alpha>\left(\frac{2p_{0}}{n}-\frac{4(n+4)}{(n-2)(n+2)}\right)^{-1}$ and $N=\frac{n+2}{n-2}$, we obtain
\begin{equation}
\beta^{*} = \frac{\frac{(n+2)p_{0}}{n}-N+1}{N+\frac{1}{\alpha}-1}
= \frac{\frac{(n+2)p_{0}}{n}-\frac{4}{n-2}}{\frac{1}{\alpha}+\frac{4}{n-2}} > \frac{n+2}{2},
\end{equation}
which implies $\beta <\frac{n+2}{n}$. Applying the iteration method and Step 1, we obtain
\begin{equation}
\sup_{M\times[0,T]}(t^{\alpha}u) = \sup_{M\times[0,T]}v
\leq C(\alpha,n,\Lambda,p_{0})\left(\int_{0}^{T}\int_{M}v^{p_{0}}d\mu_{h}dt\right)^{\frac{1}{p_{0}}}.
\end{equation}
Combining this with \eqref{Sec2: up YF eqn 2}, $v\leq u$ and $T\leq1$,
\begin{equation}
\sup_{M\times[0,T]}(t^{\alpha}u) \leq C(\alpha,n,\Lambda,p_{0})\left(\int_{0}^{T}\int_{M}u^{p_{0}}d\mu_{h}dt\right)^{\frac{1}{p_{0}}}
\leq C(\alpha,n,\Lambda,p_{0}).
\end{equation}

\bigskip

When $p_{0}=\infty$, applying the maximum principle to \eqref{evo-u} and using $T\leq1$, for any $(x,t)\in M\times[0,T]$,
\begin{equation}
u(x,t) \leq \left(C(n)\Lambda t+\sup_M u_0^\frac{4}{n-2}\right)^\frac{n-2}{4}
\leq C(n,\Lambda).
\end{equation}
\end{proof}

\subsection{More estimates along the Yamabe flow}
\begin{prop}\label{Sec2: YFC1-estimates}
Let $g(t)=u(t)^{\frac{4}{n-2}}h$ be a solution of the Yamabe flow \eqref{Yamabe flow}. If $p_0>\frac{n}{2}+\frac{2n(n+4)}{(n-2)(n+2)}$, then there is a constant $T_{0}(n,\Lambda,p_{0})$ such that $g(t)$ exists on $[0,T_{0}]$. Moreover, for all $[a,T_{0}]\subset(0,T_{0}]$, there is a constant $\lambda(a,n,\Lambda,p_0)>1$ such that
\begin{equation}\label{Sec2: YFC1-estimates eqn 3}
\left\{
\begin{array}{ll}
\lambda^{-1}h \leq g(t)\leq \lambda h;\\[1mm]
|\nabla^h g(t)|_{h} \leq \lambda;\\[1mm]
-\delta \leq R_{g(t)} \leq \lambda
\end{array}
\right.
\end{equation}
on $M\times [a,T_{0}]$.
\end{prop}
\begin{proof}
In the following, all norms are with respect to metric $h$. Denote the maximal existence time of the Yamabe flow \eqref{Yamabe flow} by $\Tmax$. First, we establish the lower bound of $\Tmax$. By \cite[Lemma 2.2]{Anderson1990}, for any $\beta\in(0,1)$, the $C^{1,\beta}$ harmonic radius of $h$ is bounded from below by $r(\beta,n,\Lambda)$. Within harmonic radius, the Laplacian operator of $h$ can be expressed as $\Delta_{h} = h^{ij}\partial_i \partial_j$, and then \eqref{evo-u} can be written as
\begin{equation}\label{harmonic-YF}
\frac{\de u}{\de t} -(n-1)u^{-\frac{4}{n-2}}h^{ij}\partial_i \partial_j u = -\frac{n-2}{4} R_h u^\frac{n-6}{n-2}.
\end{equation}
Write $T_{0}=\min(\hat{T},1)$, where $\hat{T}$ is the constant in Proposition \ref{Sec2: low YF}. We will show $\Tmax>T_{0}$. If $\Tmax\leq T_{0}$, then by Proposition \ref{Sec2: low YF} and \ref{Sec2: up YF}, for any $a\in(0,\Tmax)$,
\begin{equation}\label{Sec2: YFC1-estimates eqn 4}
C^{-1}(a,n,\Lambda,p_0) \leq \inf_{M\times[\frac{a}{3},\Tmax)} u \leq \sup_{M\times[\frac{a}{3},\Tmax)} u \leq C(a,n,\Lambda,p_0).
\end{equation}
This shows the equation \eqref{harmonic-YF} is uniformly parabolic on $[\frac{a}{3},\Tmax)$ with $L^\infty$ inhomogeneous term. Since $h$ is uniformly equivalent to the Euclidean metric $g_{\mathrm{euc}}$ in the harmonic coordinate system, we obtain uniform $C^{\delta,\frac{\delta}{2}}$ estimate on a slightly smaller ball by \cite{KrylovSafonov1980} for some $\delta(a,n,\Lambda,p_0)$. The covering argument shows
\begin{equation}\label{Sec2: YFC1-estimates eqn 1}
\|u\|_{C^{\delta,\frac{\delta}{2}}(M\times[\frac{a}{2},\Tmax))} \leq C(a,n,\Lambda,p_0).
\end{equation}
Combining this with the parabolic Schauder estimates \cite{Friedman1964}, we obtain the higher order estimates of $u$ on $[a,T_{max})$, which contradicts with the definition of $\Tmax$. Then we obtain $\Tmax>T_{0}$.

Next, we establish the required estimates. The estimate $\lambda^{-1}h\leq g(t)\leq\lambda h$ follows from \eqref{Sec2: YFC1-estimates eqn 4}. Within $C^{1,\beta}$ harmonic radius, $h$ and $g_{\mathrm{euc}}$ are uniformly comparable in $C^{1,\beta}$. Using \eqref{Sec2: YFC1-estimates eqn 1}, $|R_{h}|\leq n\Lambda$, the parabolic $L^{p}$ estimate \cite{Friedman1964} and Sobolev embedding theorem, for $p>1$,
\begin{equation}\label{Sec2: YFC1-estimates eqn 2}
\|u\|_{C^{1}} \leq C(a,n,\Lambda,\delta,p_0), \quad \|u\|_{W^{2,p}} \leq C(p,a,n,\Lambda,\delta,p_0),
\end{equation}
on $[\frac{a}{2},T_{0}]$. This shows $|\nabla^h g(t)|_{h}\leq\lambda$.

The scalar curvature $R_{g(t)}$ satisfies (see \cite[Lemma 2.2]{Chow1992}),
\begin{equation}\label{evo-R}
\left(\frac{\de}{\de t} -(n-1)\Delta_{g(t)}\right) R_{g(t)}=R_{g(t)}^2.
\end{equation}
The lower bound $R_{g(t)}\geq-\delta$ follows from the minimum principle. For the upper bound of $R_{g(t)}$ on $[a,T_{0}]$, \eqref{Sec2: YFC1-estimates eqn 2} shows for $p>1$ and $t\in [\frac{a}{2},T_0]$,
\begin{equation}
\|R_{g(t)}\|_{L^{p}} \leq C(p,a,n,\Lambda,\delta,p_0).
\end{equation}
Combining this with \eqref{Sec2: YFC1-estimates eqn 1}, the equation \eqref{evo-R} has $C^{\delta}$ parabolic coefficient and $L^{\frac{p}{2}}$ inhomogeneous term for any $p>1$ in the local coordinate system. Then $R_{g(t)}\leq\lambda$ on $M\times [a,T_0]$ follows from the parabolic $L^{p}$ estimate \cite{Friedman1964} and Sobolev embedding theorem.
\end{proof}

Lastly, we will establish the weak convergence of $g(t)$ as $t\to 0$.
\begin{prop}\label{Prop: Lp}
Let $g(t)=u(t)^{\frac{4}{n-2}}h$ be a solution of the Yamabe flow \eqref{Yamabe flow} on $[0,T_{0}]$, where $T_{0}$ is the constant in Proposition \ref{Sec2: YFC1-estimates}. If
$\frac{n}{2}+\frac{2n(n+4)}{(n-2)(n+2)}<p_{0}<\infty$, then for any $\left(\frac{2p_{0}}{n}-\frac{4(n+4)}{(n-2)(n+2)}\right)^{-1}<\alpha<1$, there is a constant $C(\alpha,n,\Lambda,p_0)$ such that
\begin{equation}
\big|\mathrm{Vol}(M,g(t))-\mathrm{Vol}(M,g_0)\big| \leq Ct^{1-\alpha}.
\end{equation}
If $p_{0}=\infty$, then for any $p>0$, there is a constant $C(p,n,\Lambda)$ such that
\begin{equation}
\int_M \big|g(t)-g(0)\big|_h^p d\mu_h \leq C(t+t^{p}).
\end{equation}
\end{prop}
\begin{proof}
When $\frac{n}{2}+\frac{2n(n+4)}{(n-2)(n+2)}<p_{0}<\infty$, along the Yamabe flow \eqref{Yamabe flow}, using $R_{g(t)}\geq-\delta\geq-1$ (see \eqref{Sec2: YFC1-estimates eqn 3}), we obtain
\begin{equation}
\frac{d}{dt}\Vol(M,g(t)) = -\frac{n}{2}\int_{M}R_{g(t)}d\mu_{g(t)} \leq \frac{n}{2}\Vol(M,g(t)).
\end{equation}
It then follows that
\begin{equation}\label{Prop3: vol-up}
\mathrm{Vol}(M,g(t)) \leq e^{\frac{nt}{2}}\mathrm{Vol}(M,g_0)
\leq \mathrm{Vol}(M,g_0)+C(n)t\mathrm{Vol}(M,g_0).
\end{equation}
By volume comparison theorem, we have $\Vol(M,h)\leq C(n,\Lambda)$. Combining this with $\|u_{0}\|_{L^{p_{0}}}\leq\Lambda$ and $p_{0}>\frac{2n}{n-2}$, we obtain
\begin{equation}\label{Prop3: vol-up 1}
\mathrm{Vol}(M,g_0) = \int_{M}u_{0}^{\frac{2n}{n-2}}d\mu_{h} \leq C(n,\Lambda,p_{0})
\end{equation}
and so
\begin{equation}\label{Prop3: vol-up 2}
\mathrm{Vol}(M,g(t)) \leq \mathrm{Vol}(M,g_0)+C(n,\Lambda,p_{0})t .
\end{equation}
It suffices to estimate the lower bound. By Proposition \ref{Sec2: low YF}, \ref{Sec2: up YF} and $R_{g(t)}\geq-\delta\geq-1$ (see \eqref{Sec2: YFC1-estimates eqn 3}),
\begin{equation}
\begin{split}
\frac{d}{dt}\mathrm{Vol}(M,g(t))
&=-\frac{n}{2}\int_M  u^{N+1}R_{g(t)}d\mu_{h}\\
&\geq -\frac{C}{t^\a} \int_M u^N(R_{g(t)}+1) d\mu_h\\
&=-\frac{C}{t^\a} \int_M u^N R_{g(t)} d\mu_h-\frac{C}{t^\a}\int_M u^{-1}d\mu_{g(t)}\\
&\geq -\frac{C}{t^\a} \int_M \left(R_hu-C(n)\Delta_h u \right) d\mu_h-\frac{C}{t^\a}\mathrm{Vol}(M,g(t))\\
&\geq -\frac{C}{t^\a} \left(\mathrm{Vol}(M,g(t))\right)^\frac{n-2}{2n}-\frac{C}{t^\a}\mathrm{Vol}(M,g(t)),
\end{split}
\end{equation}
where $C(\alpha,n,\Lambda,p_0)>0$. Using \eqref{Prop3: vol-up 1} and \eqref{Prop3: vol-up 2}, we obtain $\mathrm{Vol}(M,g(t))\leq C(n,\Lambda,p_{0})$, and then
\begin{equation}
\frac{d}{dt}\mathrm{Vol}(M,g(t)) \geq -\frac{C(\alpha,n,\Lambda,p_{0})}{t^{\alpha}}.
\end{equation}
This implies
\begin{equation}\label{Prop3: vol-low}
\mathrm{Vol}(M,g(t)) \geq \mathrm{Vol}(M,g_0) -C(\alpha,n,\Lambda,p_0)t^{1-\alpha}
\end{equation}
Combining \eqref{Prop3: vol-up 2} and \eqref{Prop3: vol-low}, we obtain
\begin{equation}
\big|\mathrm{Vol}(M,g(t))-\mathrm{Vol}(M,g_0)\big| \leq (\alpha,n,\Lambda,p_{0})t^{1-\alpha}.
\end{equation}

\bigskip

When $p_{0}=\infty$, for $t\in[0,T_{0}]$, Proposition \ref{Sec2: low YF} and \ref{Sec2: up YF} show
\begin{equation}\label{Prop3: eqn 1}
C(n,\Lambda)^{-1} \leq u(t) \leq C(n,\Lambda), \quad C(n,\Lambda)^{-1}h \leq g(t) \leq C(n,\Lambda)h.
\end{equation}
Using $R_{g(t)}\geq-\delta\geq-1$ (see \eqref{Sec2: YFC1-estimates eqn 3}), we compute
\begin{equation}
\begin{split}
\int_M|R_{g(t)}| d\mu_h \leq {} & \int_M(R_{g(t)}+1) d\mu_h+C \\
\leq {} & C\int_{M}u^{N}(R_{g(t)}+1) d\mu_h+C \\
\leq {} & C\int_{M}\left(R_{h}u-C(n)\Delta_{h}u\right)d\mu_h+C \\[1.5mm]
\leq {} & C(n,\Lambda).
\end{split}
\end{equation}
Therefore,
\begin{equation}\label{L1}
\begin{split}
\int_M \big|g(t)-g(0)\big|_h d\mu_h =& \int_M\left| \int^t_0 \frac{\de g}{\de s} ds  \right|_{h} d\mu_h \\
\leq & C\int^t_0\int_M  |R_{g(s)}|d\mu_hds \\
\leq & C(n,\Lambda)t.
\end{split}
\end{equation}
For $p\in(0,1)$, the Holder inequality shows
\begin{equation}
\int_M \big|g(t)-g(0)\big|_h^p d\mu_h \leq C\left(\int_M \big|g(t)-g(0)\big|_h d\mu_h\right)^{p} = C(p,n,\Lambda)t^{p}.
\end{equation}
For $p>1$, using \eqref{Prop3: eqn 1},
\begin{equation}
\int_M \big|g(t)-g(0)\big|_h^p d\mu_h
\leq C\int_M \big|g(t)-g(0)\big|_h d\mu_h \leq C(p,n,\Lambda)t.
\end{equation}
\end{proof}

\section{Stability on torus}\label{stability-T}
In this section, we will give the proof of Theorem~\ref{Sec1: maintheorem} and \ref{Sec1: maintheorem2}. Indeed, we will consider a slightly more general case where $M$ is a closed manifold with non-positive Yamabe invariant. Recall that 
\begin{equation}
\sigma(M)=\sup\left\{ \mathcal{Y}(M,[g]): [g] \textit{ is a conformal class of metrics on } M\right\},
\end{equation}
where 
\begin{equation}
\mathcal{Y}(M,[g_0])=\inf\left\{ \int_M R(g)\;d\mu_g: g\in[g_0],\; \mathrm{Vol}(M,g)=1\right\}.
\end{equation}

It is well known that if a smooth metric on a
compact manifold attains the Yamabe invariant and if the invariant is nonpositive, then the metric is Einstein \cite{Schoen1989}. In particular, if $g$ is a smooth metric with $R_g\geq 0$ on $M$, then $\Ric(g)\equiv 0$.
\begin{rem}\label{torus-Yama}
For $n\geq 3$, it is in general difficult to compute the Yamabe invariant. One can usually show that $\sigma(M)=0$ by proving the non-existence of metrics with positive scalar curvature. For example, Schoen-Yau \cite{SchoenYau1979,SchoenYau1979-2} and Gromov-Lawson \cite{GromovLawson1980} showed that torus cannot admit metric with positive scalar curvature and hence $\sigma(\mathbb{T}^n)=0$.
\end{rem}
 In the following, we will consider a sequence of $g_{i,0}=u_i^\frac4{n-2}h_i$ on $M^n$ with $\sigma(M)\leq 0$  where $h_i$ and $g_{i,0}$ satisfy assumptions {\bf (A)} and {\bf (B)} with $\delta=i^{-1}\rightarrow 0$ and $p_0$ sufficiently large.  Our goal is to show that $g_{i,0}$ will converge to a Ricci-flat metric on $M$ in a weak sense. Let $g_{i}(t)$ be the Yamabe flow with initial metric $g_{i,0}$.

\begin{thm}\label{Sec3: Conv-1}
Let $M^n$ be a closed manifold with $\sigma(M)\leq 0$. 
Suppose $g_{i,0}$ is a sequence of metrics on $M$ such that
\begin{enumerate}\setlength{\itemsep}{1mm}
\item[(a)] $g_{i,0}=u_i^\frac4{n-2}h_i$ for some metric $h_i$ on $M_i$ satisfying assumption {\bf (A)};
\item[(b)] $g_{i,0}$ satisfies assumption {\bf (B)} for $\delta=i^{-1}\rightarrow 0$ and $p_0>\frac{n}{2}+\frac{2n(n+4)}{(n-2)(n+2)}$.
\end{enumerate}
Then the Yamabe flow $g_{i}(t)$ with initial metric $g_{i,0}$ exists on $M\times [0,T_{0}]$, where $T_{0}(n,\Lambda,p_{0})$ is the constant in Proposition \ref{Sec2: YFC1-estimates}. Moreover, there is a sequence of diffeomorphisms $\Phi_i$ of $M$ and a Ricci-flat metric $g_\infty$ on $M$ such that after passing to subsequence, $\Phi_{i}^{*}g_{i}(t)$ converges to $g_\infty$ on $M$ in $C^{0}_{loc}(M\times (0,T_0])$.
\end{thm}

\begin{proof}
In the proof, for notational convenience all convergent sequence means convergent subsequence. By \cite{Anderson1990}, we can find a sequence of diffeomorphism $\Phi_i$ such that $\Phi_i^*h_i$ converges to some $C^{1,\beta}$ metric $h_\infty$ on $M$ in $C^{1,\gamma}$ for all $\gamma<\beta<1$ after passing to subsequence. 
For notational convenience, we will pull-back all $g_{i,0}$ by $\Phi_i$ and omit $\Phi_i^*$.

Applying Proposition \ref{Sec2: YFC1-estimates} to each $g_{i}(t)$, we obtain a sequence of Yamabe flows $g_i(t)$ on $M^{n}\times[0,T_{0}]$ which is uniform bounded in $C^1$ on any $[a,T_{0}]\subset(0,T_{0}]$. Our goal is to show that $g_i(t)$ converges to a Ricci-flat metric on $M$ weakly. In case $h_i$ is uniformly regular in $C^\infty$, it is not difficult to see that $g_i(t)$ converges to a limiting solution of the Yamabe flow $g_\infty(t)$ for $t>0$ after passing to subsequence. Since $h_i$ is only mildly regular, it is not clear to us whether $g_\infty(t)$ exists smoothly although the flow is expected to be static. Instead, we will regularize it further using the Ricci flow. 

\begin{claim}\label{Sec3: YF-limit1}
There is a Ricci-flat metric $g_{\infty}$ on $M$ such that $g_{i}(T_0)\rightarrow g_{\infty}$ in $C^0(M)$ as $i\rightarrow\infty$ after passing to subsequence and pulling back by a sequence of diffeomorphism.
\end{claim}
\begin{proof}[Proof of Claim~\ref{Sec3: YF-limit1}]
Since $g_i(T_0)$ is uniformly bounded in $C^1$ with respect to $h_i$, $g_i(T_0)$ is uniformly equivalent to $h_i$ and $h_i\rightarrow h_\infty$ in $C^{1,\gamma}$. By passing to subsequence, we may assume $g_i(T_0)\rightarrow g_\infty(T_0)$ in $C^{\gamma}$ for all $\gamma\in(0,1)$. In particular, for $\ve>0$ sufficiently small, there is $N\in \mathbb{N}$ such that for all $i>N$,
\begin{equation}
(1-\e)\bar h\leq g_i(T_0)\leq (1+\e) \bar h,\quad |\nabla^{\bar h}g_i(T_0)|_{\bar h} \leq C(\bar{h},n,\Lambda,p_{0}),
\end{equation}
where $\bar h=g_N(T_0)$. Note that the constant $C$ depends on $\bar h$, but is independent of $i$. In the following, all constants may possibly depend on $\bar h$, but are independent of $i$.

Let $\tilde g_i(s)$ be Ricci flow starting from $g_{i}(T_{0})$. By \cite[Lemma 4.3]{ShiTam2018} with $\delta=0$ (see also \cite{Simon2002}), there is a constant $S_{0}(\bar{h},n,\Lambda,p_{0})>0$ such that $\tilde g_i(s)$ exists on $M^{n}\times[0,S_{0}]$ and satisfies
\begin{equation}\label{Sec3: Conv-1 eqn 1}
|\Rm(\tilde g_i(s))|\leq C(\bar{h},n,\Lambda,p_{0})s^{-1/2}
\end{equation}
provided that $\e$ is sufficiently small depending only on $n$. In particular, $\tilde g_i(s)$ is uniformly equivalent to $\tilde g_i(0)$ and hence $h_i$ for all $s\in [0,S_0]$ by integration on time. Here we have used the estimate of the Ricci-Deturck flow with reference metric $\bar h$ from \cite[Lemma 4.3]{ShiTam2018} and the fact that the Ricci-Deturck flow is diffeomorphic to the Ricci flow with the same initial data.


Since $\mathrm{inj}(h_i)$ is bounded from below, we have a uniform lower bound on $\mathrm{inj}(\tilde g_i(s))$ thanks to the uniform equivalence of metrics. Together with Shi's estimates \cite{Shi1989} (see also \cite[Theorem 4.3]{Simon2002}) and Hamilton's compactness \cite{Hamilton1995}, we can pass $\tilde g_{i}(s)\rightarrow \tilde g_\infty(s)$ in $C^\infty$ Cheeger-Gromov sense on $M\times (0,S_{0}]$ after passing to subsequence. More precisely, there is a sequence of diffeomorphism $\Psi_i$ of $M$ such that $\Psi_i^*\tilde g_i(s)$ converges to $g_\infty(s)$ in $C^\infty_{loc}\left(M\times (0,S]\right)$. As usual, we will pull-back each metrics by  $\Psi_i$ and therefore will omit $\Psi_i^*$. We note that in this way, $\{\Psi_i^*h_i\}_{i=1}^\infty$ are not necessarily compact in $C^{1,\gamma}$ topology anymore due to the additional pull-back of diffeomorphism but it will remain uniformly equivalent to the Ricci flow $\tilde g_i(s)$ thanks to the curvature estimates in \eqref{Sec3: Conv-1 eqn 1}.

It is well-known that Ricci flow preserved the lower bound of the scalar curvature, we have
\begin{equation}\label{Sec3: Conv-1 eqn 2}
R_{\tilde g_i(s)} \geq R_{\tilde g_i(0)} = R_{g_i(T_{0})} \geq -i^{-1},
\end{equation}
where we used \eqref{Sec2: YFC1-estimates eqn 3} in the last inequality. Letting $i\rightarrow\infty$, we obtain $R_{\tilde g_\infty(s)}\geq 0$ for $s\in (0,S_{0}]$. By the assumption of $\sigma(M)\leq 0$ and uniqueness of the Ricci  flow, $\tilde g_\infty(s)\equiv g_{\infty}$ on $M$ for $s\in (0,S_{0}]$ for some Ricci-flat metric $g_\infty$. Combining \eqref{Sec3: Conv-1 eqn 1} and the equation of Ricci-flow 
\begin{equation}
\frac{\de \tilde{g}_{kl}}{\de s} =-2\Ric(\tilde{g})_{kl}
\end{equation}
we obtain
\begin{equation}
\big|\tilde{g}_{i}(s)-g_{i}(T_{0})\big|_{\tilde g_i(S_0)}
= \big|\tilde{g}_{i}(s)-\tilde{g}_{i}(0)\big|_{\tilde g_i(S_0)} \leq C(\bar{h},n,\Lambda,p_{0})s^{\frac{1}{2}}.
\end{equation}
Combining this with $\tilde g_{i}(s)\rightarrow g_{\infty}$ in $C^{\infty}_{\mathrm{loc}}$ on $(0,S_{0}]$, we conclude $g_{i}(T_0)=\tilde g_i(0)\rightarrow g_{\infty}$ in $C^0(M)$ as $i\rightarrow\infty$.
\end{proof}

Next, we claim that $g_i(t)$ converges to the same Ricci-flat metric $g_{\infty}$.
\begin{claim}\label{Sec3: YF-limit}
For all $t\in (0,T_{0}]$, $g_{i}(t)\rightarrow g_{\infty}$ in $C^0$ as $i\rightarrow\infty$, where $g_{\infty}$ is the Ricci-flat metric obtain from Claim~\ref{Sec3: YF-limit1}.
\end{claim}
\begin{proof}
By Proposition \ref{Sec2: YFC1-estimates}, for $a\in (0,T_{0}]$, $g_i(t)$ is uniformly equivalent to $g_i(T_0)$ on $[a,T_{0}]$ and $-i^{-1}\leq R_{g_{i}(t)}\leq\lambda$. Using \eqref{evo-R}, we compute
\begin{equation}
\begin{split}
 \frac{\de}{\de t}\left(\int_{M} R_{g_{i}(t)}d\mu_{g_{i}(t)}\right) =& \left(1-\frac{n}{2}\right)\int_{M} R^2_{g_{i}(t)} d\mu_{g_i(t)}\\
\geq {} & -C(a,n,\Lambda,p_{0})\left(\int_{M} R_{g_{i}(t)} d\mu_{g_{i}(t)}+i^{-1}\right).
\end{split}
\end{equation}
Hence, for all $t\in [a,T_{0}]$,
\begin{equation}\label{Sec3: R-flat-YF}
\int_{M} R_{g_i(t)}d\mu_{g_i(t)}\leq C(a,n,\Lambda,p_{0})\left(\int_{M} R_{g_i(T_{0})}d\mu_{g_i(T_{0})}+i^{-1}\right).
\end{equation}
We now estimate the integral of scalar curvature on the right hand side. We will make use of the smooth convergence of the Ricci-flow. 
The scalar curvature $R_{\tilde{g}_{i}(s)}$ satisfies (see e.g., \cite[(2.5.5)]{Topping2006})
\begin{equation}
\left(\frac{\de}{\de s}-\Delta_{\tilde{g}_{i}(s)}\right)R_{\tilde{g}_{i}(s)} = 2\left|\Ric(\tilde g_i(s))\right|_{\tilde{g}_{i}(s)}^2.
\end{equation}
By \eqref{Sec3: Conv-1 eqn 1} and \eqref{Sec3: Conv-1 eqn 2},
\begin{equation}
-i^{-1}\leq R_{\tilde g_i(s)}\leq C(\bar{h},n,\Lambda,p_{0})s^{-\frac{1}{2}}.
\end{equation}
Then along the Ricci flow, we have
\begin{equation}
\begin{split}
 \frac{\de}{\de s} \left( \int_{M} R_{\tilde g_i(s)}d\mu_{\tilde g_i(s)}\right)
= &\int_{M}\left(2\left|\Ric(\tilde g_i(s))\right|_{\tilde g_i(s)}^2-|R_{\tilde g_i(s)}|^2\right)d\mu_{\tilde g_i(s)}\\
\geq  & -C(a,\bar h,n,\Lambda,p_{0})s^{-\frac{1}{2}}\left(\int_{M} R_{\tilde g_i(s)}d\mu_{\tilde g_i(s)}+i^{-1}\right).
\end{split}
\end{equation}
This shows
\begin{equation}
\begin{split}
 \int_{M} R_{g_i(T_{0})}d\mu_{g_i(T_{0})} =& \int_{M} R_{\tilde g_i(0)}d\mu_{\tilde g_i(0)} \\
\leq {} & C(a,\bar h,n,\Lambda,p_{0})\left(\int_{M} R_{\tilde g_i(S_{0})}d\mu_{\tilde g_i(S_{0})}+i^{-1}\right).
\end{split}
\end{equation}
Using \eqref{Sec3: R-flat-YF}, for all $t\in [a,T_{0}]$,
\begin{equation}
\int_{M} R_{g_i(t)}d\mu_{g_i(t)}
\leq C(a,\bar{h},n,\Lambda,p_0) \left(\int_{M} R_{\tilde g_i(S_{0})}d\mu_{\tilde g_i(S_{0})}+i^{-1}\right).
\end{equation}
Combining this with $R_{g(t)}\geq-i^{-1}$ (see \eqref{Sec2: YFC1-estimates eqn 3}), we have
\begin{equation}
\int_{M}|R_{g_i(t)}|d\mu_{g_i(t)}
\leq C(a,\bar{h},n,\Lambda,p_0) \left(\int_{M} R_{\tilde g_i(S_{0})}d\mu_{\tilde g_i(S_{0})}+i^{-1}\right).
\end{equation}
By the similar computation of \eqref{L1} and the fact that $g_i(t)$ is uniformly equivalent to $h_i$ and hence $\tilde g_i(S_0)$, we have 
\begin{equation}
\begin{split}
\int_{M} \big|g_i(t)-g_i(T_{0})\big|_{g_i(T_0)}d\mu_{g_i(T_0)} \leq & C\int_{a}^{T_{0}} \int_M |R_{g_i(s)}| d\mu_{g_i(T_0)}ds \\
\leq {} & C(a,\bar{h},n,\Lambda,p_0) \left(\int_{M} R_{\tilde g_i(S_{0})}d\mu_{\tilde g_i(S_{0})}+i^{-1}\right).
\end{split}
\end{equation}
Using Claim \ref{Sec3: YF-limit1}, the smooth convergence of $\tilde g_{i}(S_{0})$ to $g_{\infty}$ and the fact that $\Ric(g_\infty)=0$, we see that $g_{i}(t)\rightarrow g_{\infty}$ in $L^1(M,g_\infty)$ as $i\rightarrow\infty$ for all $t\in [a,T_{0}]$. By the uniform $C^1$ estimates from Proposition \ref{Sec2: YFC1-estimates}, we can improve the convergence to $C^0$. This completes the proof since $a\in(0,T_{0}]$ is arbitrary.
\end{proof}
\end{proof}

We now prove the stability Theorem for manifolds with $\sigma(M)\leq 0$. This will imply Theorem \ref{Sec1: maintheorem} and \ref{Sec1: maintheorem2} when $M=\mathbb{T}^n$ by Remark~\ref{torus-Yama} as Ricci-flat metric on torus is flat from the splitting Theorem of Cheeger-Gromov \cite{CheegerGromov1971}.
\begin{cor}
Under the assumption in Theorem \ref{Sec3: Conv-1}, if $\mathrm{diam}(M,g_{i,0})$ are uniformly bounded from above, then after passing to subsequence, $g_{i,0}$ converges to a Ricci-flat metric $g_{\infty}$ in the volume preserving intrinsic flat sense. 
\end{cor}
\begin{proof}
By \eqref{Sec2: YFC1-estimates eqn 3}, the Yamabe flow $g_i(t)$ in Theorem \ref{Sec3: Conv-1} satisfies
\begin{equation}
\frac{\de g_{i}}{\de t} = -R_{g_{i}}g_{i} \leq i^{-1}g_{i},
\end{equation}
which implies
\begin{equation}
g_i(t) \leq e^{-i^{-1}t}g_{i}(0) = e^{-i^{-1}t}g_{i,0}.
\end{equation}
In particular,
\begin{equation}
g_i(T_{0}) \leq e^{-i^{-1}T_{0}}g_{i,0}.
\end{equation}
By Theorem \ref{Sec3: Conv-1}, $F_i^*g_i(T_0)$ converges to a Ricci-flat metric $g_{\infty}$ in $C^0(M)$ where $F_i=\Phi_i\circ \Psi_i$. By relabelling the index, we may assume without loss of generality that
\begin{equation}\label{proof of maintheorem eqn 1}
\left(1-\frac1i\right)g_{\infty}\leq F_i^*g_{i,0}.
\end{equation}
The intrinsic flat convergence follows from \eqref{proof of maintheorem eqn 1}, volume convergence in Proposition \ref{Prop: Lp} and \cite[Theorem 2.1]{AllenPeralesSormani2020}.
\end{proof}

\begin{cor}[Theorem \ref{Sec1: maintheorem2}]
Under the assumption in Theorem~\ref{Sec3: Conv-1}, if $u_i$ are uniformly bounded from above, then there is a sequence of diffeomorphisms $F_i$ of $M$ such that after passing to subsequence, $F_i^*g_{i,0}$ converges to a Ricci-flat metric $g_{\infty}$ in $L^p(M,g_{\infty})$ for all $p>0$. Moreover, $g_{i,0}$ converges to $g_\infty$ in the measured Gromov-Hausdorff sense.
\end{cor}
\begin{proof}
The $L^p$ convergence follows from Proposition \ref{Prop: Lp} and Theorem \ref{Sec3: Conv-1}. As shown in proof of Theorem \ref{Sec3: Conv-1}, $F_i^*h_{i}\leq C(n,\Lambda)\cdot F_i^*g_{i}(T_0)$ and $F_i^*g_{i}(T_0)$ converges to $g_\infty$ in $C^0(M)$. Combining this with \eqref{proof of maintheorem eqn 1}, we have 
\begin{equation}\label{proof of maintheorem2 eqn 1}
\left(1-\frac1i\right)g_{\infty} \leq F_i^*g_{i,0}\leq C(n,\Lambda)g_{\infty}.
\end{equation}
for $i$ sufficiently large. 
The measured Gromov-Hausdorff convergence follows from \eqref{proof of maintheorem2 eqn 1}, the $L^p$ convergence and \cite[Theorem 1.2]{AllenSormani2020}.
\end{proof}

\end{document}